\documentclass[11pt]{article}

\usepackage{amsfonts}
\usepackage{amsmath}

\newcommand{\beq}{\begin{equation}}
\newcommand{\eeq}{\end{equation}}
\newcommand{\bea}{\begin{eqnarray}}
\newcommand{\eea}{\end{eqnarray}}
\newcommand{\beas}{\begin{eqnarray*}}
\newcommand{\eeas}{\end{eqnarray*}}

\newtheorem{theorem}{Theorem}[section]

\newtheorem{definition}[theorem]{Definition}
\newtheorem{proposition}[theorem]{Proposition}

\newtheorem{lemma}[theorem]{Lemma}

\newtheorem{remark}[theorem]{Remark}
\newtheorem{example}[theorem]{Example}
\newtheorem{examples}[theorem]{Examples}
\newtheorem{foo}[theorem]{Remarks}

\newenvironment{proof}{\addvspace{\medskipamount}\par\noindent{\it
Proof}.}
{\unskip\nobreak\hfill$\Box$\par\addvspace{\medskipamount}}

\newcommand{\bM}{\mathbb M}

\newcommand{\F}{\mathcal F}

\newcommand{\M}{\mathbb M}

\title{Sobolev, Poincar\'e and isoperimetric  inequalities  for subelliptic diffusion operators satisfying a generalized curvature dimension inequality }
\author{Fabrice Baudoin\footnote{First author supported in part by
NSF Grant DMS 0907326}, Bumsik Kim }

\date{Department of Mathematics, Purdue University \\
 West Lafayette, IN, USA}

\begin{document}
\maketitle

\begin{abstract}

By adapting some ideas of  M. Ledoux \cite{ledoux2}, \cite{ledoux-stflour} and  \cite{Led}  to a sub-Riemannian framework we study  Sobolev, Poincar\'e and isoperimetric inequalities associated to subelliptic diffusion operators that satisfy the generalized curvature dimension inequality that was  introduced by F. Baudoin and N. Garofalo in  \cite{Bau2}.
Our results apply in particular on all CR Sasakian manifolds whose horizontal Webster-Tanaka-Ricci curvature is non negative, all Carnot groups with step two, and wide subclasses of principal bundles over Riemannian manifolds whose Ricci curvature is non negative.
\end{abstract}

\baselineskip 0.30in

\tableofcontents

\section{Introduction and framework} \label{sec:framework}

In this paper, $\bM$ will be a $C^\infty$ connected finite dimensional manifold endowed with a smooth measure $\mu$ and a second-order diffusion operator $L$ on $\M$, locally subelliptic in the sense of \cite{FP1} (see also \cite{JSC}), satisfying $L1=0$, and
\begin{equation*}
\int_\bM f L g d\mu=\int_\bM g Lf d\mu,\ \ \ \ \ \ \int_\bM f L f d\mu \le 0,
\end{equation*}
for every $f , g \in C^ \infty_0(\bM)$.
We indicate with $\Gamma(f):=\Gamma(f,f)$ the \textit{carr\'e du champ} of $L$, that is the quadratic differential form defined by
\begin{equation}\label{gamma}
\Gamma(f,g) =\frac{1}{2}(L(fg)-fLg-gLf), \quad f,g \in C^\infty(\bM).
\end{equation}

There is an intrinsic distance associated to $L$ that can be defined via the notion of subunit curves (see \cite{FP1}). An absolutely continuous curve $\gamma: [0,T] \rightarrow \bM$ is said to be subunit for the operator $L$ if for every smooth function $f : \bM \to \mathbb{R}$ we have $ \left| \frac{d}{dt} f ( \gamma(t) ) \right| \le \sqrt{ (\Gamma f) (\gamma(t)) }$.  We then define the subunit length of $\gamma$ as $\ell_s(\gamma) = T$. Given $x, y\in \M$, we indicate with
\[
S(x,y) =\{\gamma:[0,T]\to \M\mid \gamma\ \text{is subunit for}\ L, \gamma(0) = x,\ \gamma(T) = y\}.
\]
In this paper we assume that $S(x,y)$ is not empty for every $ x, y\in \M$. Under such assumption  it is easy to verify that
\begin{equation}\label{ds}
d(x,y) = \inf\{\ell_s(\gamma)\mid \gamma\in S(x,y)\},
\end{equation}
defines a true distance on $\M$. Furthermore, in that case, it is known that
\begin{equation}\label{di}
d(x,y)=\sup \left\{ |f(x) -f(y) | \mid f \in  C^\infty(\bM) , \| \Gamma(f) \|_\infty \le 1 \right\},\ \ \  \ x,y \in \bM.
\end{equation}
Throughout this paper, we assume that the metric space $(\M,d)$ is complete.

In addition to the differential form \eqref{gamma}, we assume that $\M$ is endowed with another smooth symmetric bilinear differential form, indicated with $\Gamma^Z$, satisfying for $f,g \in C^\infty(\M)$
\[
\Gamma^Z(fg,h) = f\Gamma^Z(g,h) + g \Gamma^Z(f,h),
\]
and $\Gamma^Z(f) = \Gamma^Z(f,f) \ge 0$.

We make the following assumptions that will be in force throughout the paper:

\

\begin{itemize}
\item[(H.1)] There exists an increasing
sequence $h_k\in C^\infty_0(\bM)$   such that $h_k\nearrow 1$ on
$\bM$, and \[
||\Gamma (h_k)||_{\infty} +||\Gamma^Z (h_k)||_{\infty}  \to 0,\ \ \text{as} \ k\to \infty.
\]
\item[(H.2)]
For any $f \in C^\infty(\bM)$ one has
\[
\Gamma(f, \Gamma^Z(f))=\Gamma^Z( f, \Gamma(f)).
\]
\item[(H.3)]
For every $f \in C_0^\infty(\bM)$ and $T \ge 0$, one has 
\[
\sup_{t \in [0,T]} \| \Gamma(P_t f)  \|_{ \infty}+\| \Gamma^Z(P_t f) \|_{ \infty} < +\infty,
\]
where $P_t$ is the heat semigroup generated by $L$.

 \end{itemize}
 \

As it has been proved in \cite{Bau2}, the assumption (H.1)  implies in particular that $L$ is essentially self-adjoint on $C^\infty_0(\bM)$. The assumption (H.2) is more subtle and is crucial for the validity of most  the subsequent results: It is discussed in details in \cite{Bau2} in several geometric examples. In the sub-Riemannian geometries covered by the present work (H.2) means that the torsion of the sub-Riemannian connection is vertical. Assumption (H.3) is necessary to rigorously justify the Bakry-\'Emery type arguments. It is a consequence of the generalized curvature dimension inequality below in many examples (see \cite{Bau2}). Other sufficient conditions ensuring that (H.3) is satisfied may be found in \cite{Wang}.

In addition to $\Gamma$ and $\Gamma^Z$ we need the following second order differential bilinear forms:
\begin{equation}\label{gamma2}
\Gamma_{2}(f,g) = \frac{1}{2}\big[L\Gamma(f,g) - \Gamma(f,
Lg)-\Gamma (g,Lf)\big],
\end{equation}
\begin{equation}\label{gamma2Z}
\Gamma^Z_{2}(f,g) = \frac{1}{2}\big[L\Gamma^Z (f,g) - \Gamma^Z(f,
Lg)-\Gamma^Z (g,Lf)\big].
\end{equation}
As for $\Gamma$ and $\Gamma^Z$, we will freely use the notations  $\Gamma_2(f) = \Gamma_2(f,f)$, $\Gamma_2^Z(f) = \Gamma^Z_2(f,f)$.

The following curvature dimension condition was introduced in \cite{Bau2}.

\begin{definition}(See \cite{Bau2})
We say that $L$ satisfies the  \emph{generalized curvature-dimension inequality} \emph{CD}$(\rho_1,\rho_2,\kappa,d)$ if
there exist constants $\rho_1 \in \mathbb{R} $,  $\rho_2 >0$, $\kappa \ge 0$, and $0< d < \infty$ such that the inequality
\begin{equation*}
\Gamma_2(f) +\nu \Gamma_2^Z(f) \ge \frac{1}{d} (Lf)^2 +\left( \rho_1 -\frac{\kappa}{\nu} \right) \Gamma(f) +\rho_2 \Gamma^Z(f)
\end{equation*}
 holds for every  $f\in C^\infty(\bM)$ and every $\nu>0$, where $\Gamma_2$ and $\Gamma_2^Z$ are defined by (\ref{gamma2}) and (\ref{gamma2Z}).
\end{definition}

The motivation for such criterion comes from the study of several examples coming from sub-Riemannian geometry where the generalized curvature dimension inequality turns out to be equivalent to lower bounds on intrinsic curvature tensors (see \cite{Bau2}). The parameter $\rho_1$ is of special importance, it is the curvature parameter. The condition $\rho_1=0$ means that the ambient space has a non negative curvature whereas the condition $\rho_1>0$ means that it has a positve curvature. In particular,  in the latter case a Bonnet-Myers type theorem was proved in \cite{Bau2}, implying that $\mathbb{M}$ needs to be compact. Let us mention that in a recent work F.Y. Wang proposed an extension of the generalized curvature dimension inequality in \cite{Wang}.

\

Our goal in the present work will be to discuss Sobolev  type embeddings, isoperimetric type results and Poincar\'e inequalities by using the generalized curvature dimension inequality. Our methods will exploit and extend to the present subelliptic framework some clever and beautiful ideas due to M. Ledoux (see  \cite{ledoux2}, \cite{ledoux-stflour} and  \cite{Led} ) who used heat semigroup methods to study isoperimetric, Sobolev and Poincar\'e inequalities. Our discussion will be based on the curvature parameter $\rho_1$.

\

In the case $\rho_1=0$, which is studied in Section 2, one of our main results  is the following Besov-Sobolev embedding:

\begin{theorem}
Assume that $L$ satisfies the  \emph{generalized curvature-dimension inequality} \emph{CD}$(0,\rho_2,\kappa,d)$. For every $1 \leq p < q < \infty$ and every $ f \in W^{1,p}(\mathbb{M}) $, we have
\begin{align*}
 \| f \|_q \leq C \| \sqrt{ \Gamma(f) } \|_{p}^{\theta} \| f \|_{ B_{\infty ,\infty}^{\theta / (\theta-1)} }^{1-\theta}
\end{align*}
 where $\theta = \frac{p}{q}$,  where $C>0$ is a constant that only depends on $p,q,\rho_2,\kappa,d$ and where $\| \cdot \|_{ B_{\infty ,\infty}^{\theta / (\theta-1)} }$ is the  Besov  norm which is introduced in \eqref{Besov_norm}.
\end{theorem}
We then prove that this Besov-Sobolev embedding implies the following  isoperimetric inequality:
\begin{proposition}
Assume that $L$ satisfies the  \emph{generalized curvature-dimension inequality} \emph{CD}$(0,\rho_2,\kappa,d)$.  Assume that there exists constants $C>0$ and $D>0$ such that for every $x\in \mathbb{M}$, $R\ge 0$,  $\mu(B(x,R))\ge C R^D$.  For any $1\leq p,q,r < \infty$ with $ \frac{1}{q}=\frac{1}{p}-\frac{r}{qD} $, there exists a constant $C'>0$ such that $\forall f \in C_0^\infty(\mathbb{M})$, we have
\[
\| f \|_q \leq C' \| \sqrt{\Gamma(f)} \|_p^{p/q} ~~\|f\|_r^{1-p/q},
\]
and there exists a constant $C'' >0$ such that for every Caccioppoli set $E\subset \bM$ one has
\begin{align}\label{iso2}
\mu(E)^{\frac{D-1}{D}} \le C'' P(E),
\end{align}
where $P(E)$ denotes the horizontal perimeter of $E$ in $\bM$.
\end{proposition}
In the isoperimetric inequality (\ref{iso2}) the constant $C''$ we obtain is not sharp but the exponent $\frac{D-1}{D}$ is correct as the example of the Heisenberg group, to which the result applies, shows.  We can observe that in the Euclidean case  the optimal isoperimetric constant can be obtained from the semigroup method  by using Riesz-Sobolev rearrangement type inequalities (see \cite{BLL} and \cite{ledoux-stflour}). But, so far, to the knowledge of the authors, the rearrangement inequality is not available  in the Heisenberg group case. Since the celebrated note of Pansu \cite{Pan}, the problem of the optimal isoperimetric constant on the Heisenberg group is a long-standing open problem (see \cite{CDPT}).

\

In the Section 3, we study  the case where the curvature parameter  $\rho_1$ is positive. In that case, as we stressed it before, the manifold $\M$ needs to be compact and the measure $\mu$ finite. We obtain the following Poincar\'e inequality:

\begin{proposition}
Assume that $L$ satisfies the  \emph{generalized curvature-dimension inequality} \emph{CD}$(\rho_1,\rho_2,\kappa,d)$ with $\rho_1>0$. Let $1\leq p <\infty$. There exists $C= C_p(\rho_1,\rho_2,\kappa,d)>0$ such that for every  $ f\in C_0^\infty (\M)$,
\begin{align*}
 \| f- f_\M \|_p \leq  C  \| \sqrt{\Gamma(f)} \|_p ,
\end{align*}
where $f_\M = \frac{1}{\mu(\M)} \int_\M f d\mu $.
\end{proposition}

Interestingly, the constant $C$  we obtain is explicit enough and does not depend on $p$ in for $1 \le p <2$ or $2\le p <\infty$.  Also $C$ does not depend on the dimension $d$ when $1 \le p <2$.

The end of Section 3 is then devoted to the study of the isoperimetric constant introduced by Cheeger in \cite{Che} in a Riemannian framework and to the study of the first non zero eigenvalue of $\M$.  Concerning the Cheeger's isoperimetric constant, we prove in particular the following lower bound:

\begin{proposition}
Assume that $L$ satisfies the  \emph{generalized curvature-dimension inequality} \emph{CD}$(\rho_1,\rho_2,\kappa,d)$ with $\rho_1>0$ and that $\mu(\M)=1$. Define
\[
\iota=\inf \frac{P(E)}{\mu(E)},
\]
where the infimum runs over all Caccioppoli sets $E$ such that $\mu(E)\le \frac{1}{2}$. We have then
\[
\iota \ge \frac{1}{2}\sqrt{\frac{\rho_1}{2} } \frac{1}{  1+\frac{2\kappa}{\rho_2}  }.
\]
\end{proposition}
And concerning the first eigenvalue we prove the following analogue of the celebrated Lichnerowicz' lower bound:
\begin{proposition}
Assume that $L$ satisfies the  \emph{generalized curvature-dimension inequality} \emph{CD}$(\rho_1,\rho_2,\kappa,d)$ with $\rho_1>0$. The first non zero eigenvalue $\lambda_1$ of $-L$ satisfies the estimate
\[
\lambda_1 \ge \frac{\rho_1 \rho_2}{\frac{d-1}{d} \rho_2 +\kappa}.
\]
\end{proposition}
\

To conclude this introduction, let us now turn to the fundamental question of examples to which the above results apply. We refer the reader to \cite{Bau2} for more details about most of the examples we discuss below.

Besides Laplace-Beltrami operators on complete Riemannian manifolds with Ricci curvature bounded from below, a wide class of examples is given by sub-Laplacians on  sub-Riemannian manifold with transverse symmetries.  Sub-Laplacians on  Sasakian manifolds form a special and interesting subclasses that we quickly describe below.
Let $\mathbb{M}$ be a complete strictly pseudo convex  CR Sasakian manifold with real dimension $2n +1$. Let $\theta$ be a pseudo-hermitian form on $\mathbb{M}$ with respect to which the Levi form is positive definite. The kernel of $\theta$ determines an horizontal bundle $\mathcal H$. Denote now  by $T$ the Reeb vector field on $\bM$, i.e., the characteristic direction of $\theta$.
We recall that the CR manifold $(\bM,\theta)$ is called Sasakian if $T$ is a sub-Riemannian Killing field. For instance the standard CR structures on the  Heisenberg group $\mathbb{H}_{2n+1}$ and the sphere $\mathbb{S}^{2n+1}$ are Sasakian. On CR manifolds, there is a canonical subelliptic diffusion operator which is called the CR sub-Laplacian. It plays an analogue in CR geometry as the Laplace-Beltrami operator does in Riemannian geometry.
In this framework we have the following result that shows the relevance of the generalized curvature dimension inequality.
\begin{proposition}\cite{Bau2}
Let $(\bM,\theta)$ be a complete \emph{CR} Sasakian manifold  with real dimension $2n+1$. If
for every $x\in \bM$ the Tanaka-Webster Ricci tensor satisfies the bound
\[
\emph{Ric}_x(v,v)\ \ge \rho_1|v|^2,
\]
for every horizontal vector $v\in \mathcal H_x$, then, for the CR sub-Laplacian of $\bM$,  the curvature-dimension inequality \emph{CD}$(\rho_1,\frac{d}{4},1,d)$ holds with $d = 2n$ and $\Gamma^Z(f)=(Tf)^2$ and  the hypothesis (H.1),(H.2),(H.3) are satisfied.
\end{proposition}

In addition to sub-Laplacians on  Heisenberg groups, more generally, the sub-Laplacian on any Carnot group of step 2 has been shown to satisfy the generalized curvature-dimension inequality \emph{CD}$(0,\rho_2,\kappa,d)$, for some values of the parameters $\rho_2$ and $\kappa$.

\section{The case $\rho_1=0$}

Throughout the Section 2, we assume that $L$ satisfies the  \emph{generalized curvature-dimension inequality} \emph{CD}$(0,\rho_2,\kappa,d)$ with  $\rho_2 >0$ and $\kappa \ge 0$.

The main tool to prove the theorems mentioned in the introduction, is the heat semigroup $P_t=e^{tL}$, which is defined using the spectral theorem. Since $L$ satisfies the curvature dimension inequality,  this semigroup is stochastically complete (see \cite{Bau2}), i.e. $P_t 1 =1$. Moreover, thanks to the hypoellipticity of $L$, for $f \in L^p(\bM)$,  $1 \le p \le \infty$, the function $(t,x) \rightarrow P_t f(x)$ is
smooth on $\mathbb{M}\times (0,\infty) $ and
\[ P_t f(x)  = \int_{\mathbb M} p(x,y,t) f(y) d\mu(y)\] where $p(x,y,t) = p(y,x,t) > 0$ is the so-called heat
kernel associated to $P_t$.

 A key ingredient in the following  analysis is the following gradient bound that was proved in \cite{Bau2}.

\begin{theorem}[Li-Yau type gradient estimate with $\rho_1=0$]\label{T:ge}
 Let $f \in C_0^\infty(\bM)$, $f  \ge 0$, $f \not\equiv 0$, then the following inequality holds for $t>0$:
\[
\Gamma (\ln P_t f)  \le
\left(1+\frac{3\kappa}{2\rho_2}\right)
\frac{LP_t f}{P_t f} +\frac{d\left(
1+\frac{3\kappa}{2\rho_2}\right)^2}{2t}.
\]
\end{theorem}

\subsection{Gradient bounds for the heat semigroup}

\begin{proposition}\label{GB1}

Let  $f \in C_0^\infty(\bM)$.
\begin{itemize}
\item If $1 \le p < 2$, then for every $t >0$,
\[
\left\| \sqrt{\Gamma(P_t f)} \right\|_p  \leq \frac{ 1+\frac{3\kappa}{2\rho_2} }{ \sqrt{ 1+(p-1) \left(1+\frac{3\kappa}{2\rho_2}\right)} }\sqrt{ \frac{d}{2t}}  \| f\|_p.
\]

\item If $2 \le p \le +\infty$,  then for every $t>0$,
\[
\left\| \sqrt{\Gamma(P_t f)} \right\|_p  \leq \sqrt{ \frac{ 1+\frac{2\kappa}{\rho_2} }{2t } }   \| f\|_p.
\]
\end{itemize}
\end{proposition}

\begin{proof}
Suppose that $1\leq p < 2$.

By Theorem \ref{T:ge}, for $f\in C_0^\infty (\M)$, $f\geq0$, $f\not\equiv 0$, $t>0$,
\[
 (P_t f)^{p-2} \Gamma(P_t f) \leq \frac{D}{d} (P_t f)^{p-1} ( L  P_t f ) + \frac{D^2}{2td} (P_t f)^p ,
\]
where $D=d(1+\frac{3\kappa}{2\rho_2} )$. It follows that
\begin{align*}
 \int_\M (P_t f)^{p-2} \Gamma(P_t f) & d\mu  \leq \frac{D}{d} \int_\M  (P_t f)^{p-1} ( L P_t f ) d\mu + \frac{D^2}{2td} \int_\M (P_t f)^p d\mu \\
  & = - \frac{D}{d} \int_\M \Gamma( (P_t f)^{p-1} ,  P_t f  ) d\mu + \frac{D^2}{2td} \int_\M (P_t f)^p d\mu \\
  & = - \frac{D}{d} \int_\M (p-1) (P_t f)^{p-2} \Gamma( P_t f ) d\mu + \frac{D^2}{2td} \int_\M (P_t f)^p d\mu .
\end{align*}
Observing $ \int_\M (P_t f)^p d\mu = \| P_t f\|_p^p \leq \|f\|_p^p$, we get
\[
  \int_\M (P_t f)^{p-2} \Gamma( P_t f ) d\mu \leq \frac{1}{1+(p-1)\frac{D}{d} } \left( \frac{D^2}{2td} \right) \|f \|_p^p.
\]

On the other hand, let us pick $\alpha=\frac{p}{2} , \beta= \frac{2-p}{2}$. Since $1\leq p<2$, one can easily check that
\begin{align*}
\left( \int_\M (P_t f)^{p-2} \Gamma( P_t f ) d\mu \right)^\alpha & = \left\| (P_t f)^{\frac{p(p-2)}{2}} \Gamma( P_t f )^\frac{p}{2} \right\|_\frac{2}{p} ,\\
 \left( \int_\M (P_t f)^p d\mu \right)^\beta & = \left\| (P_t f)^\frac{p(2-p)}{2} \right\|_\frac{2}{2-p} .
\end{align*}
So, by H\"{o}lder's inequality, we obtain
\begin{align*}
 \int_\M \Gamma(P_t f)^\frac{p}{2} d\mu \leq \left( \int_\M (P_t f)^{p-2} \Gamma( P_t f ) d\mu \right)^\alpha \left( \int_\M (P_t f)^p d\mu \right)^\beta ,
\end{align*}
or equivalently,
\begin{align*}
 \int_\M (P_t f)^{p-2} \Gamma( P_t f ) d\mu &\geq \left[ \int_\M \Gamma(P_t f)^\frac{p}{2} d\mu \left( \int_\M (P_t f)^p d\mu \right)^{-\beta} \right]^\frac{1}{\alpha} \\
 & = \left[ \| \sqrt{\Gamma(P_t f)} \|_p^p \ \| P_t f \|_p^{-p\beta} \right]^\frac{1}{\alpha} \\
 & = \left[ \| \sqrt{\Gamma(P_t f)} \|_p^p \ \| P_t f \|_p^{-p(2-p)/2} \right]^\frac{2}{p} \\
 & = \| \sqrt{\Gamma(P_t f)} \|_p^2 \ \| P_t f \|_p^{-(2-p)}
\end{align*}
Therefore, for $1\leq p <2$, we obtain
\begin{align*}
 \| \sqrt{\Gamma(P_t f)} \|_p^2 & \leq \left[ \frac{1}{1+(p-1)\frac{D}{d} } \left( \frac{D^2}{2td} \right) \|f \|_p^p \right] \| P_t f \|_p^{2-p}\\
 & \leq \frac{1}{1+(p-1)\frac{D}{d} } \left( \frac{D^2}{2td} \right) \|f \|_p^p  \| f \|_p^{2-p} \\
 & = \frac{1}{1+(p-1)\frac{D}{d} } \left( \frac{D^2}{2td} \right) \|f \|_p^2 .
\end{align*}

For $f\in C_0^\infty (\M)$, let us decompose $f= f^+ - f^-$, where $f^+ = \max(f,0)$, $f^- = -\min(f,0)$.

Then for each of $f^+$ and $f^-$, the above gradient estimate holds.

We can then finish the proof by observing that $\|f \|_p = \|f^+ \|_p+\|f^- \|_p$ and
$\| \sqrt{\Gamma(P_t f)} \|_p \leq \| \sqrt{\Gamma(P_t f^+)}+\sqrt{\Gamma(P_t f^-)} \|_p \leq \| \sqrt{\Gamma(P_t f^+)} \|_p+\| \sqrt{\Gamma(P_t f^-)} \|_p$.\\

\bigskip

Now suppose that $2\leq p \leq +\infty$.

In \cite{BB}, the following reverse Poincar\'{e} inequality (Caccioppoli type inequality) is proved:
\[
 \Gamma(P_t f) + \rho_2 t \Gamma^Z (P_t f) \leq \frac{1+\frac{2\kappa}{\rho_2}}{2t}\left( P_t (f^2) - (P_t f)^2 \right) .
\]
For $2\leq p \leq +\infty$, one can write $\| P_t(f^2) \|_\frac{p}{2} \leq \| f^2 \|_\frac{p}{2} = \| f \|_p^2 $.

Therefore, we have
\begin{align*}
  \left\| \Gamma(P_t f) \right\|_\frac{p}{2} & \leq \frac{1+\frac{2\kappa}{\rho_2}}{2t}  \left\| P_t (f^2) \right\|_\frac{p}{2} \\
   &\leq \frac{1+\frac{2\kappa}{\rho_2}}{2t}  \| f \|_p^2 ,
\end{align*}
which implies
\[
 \left\| \sqrt{\Gamma(P_t f)} \right\|_p \leq \sqrt{\frac{1+\frac{2\kappa}{\rho_2}}{2t} } \| f \|_p .
\]
\end{proof}

\subsection{Pseudo-Poincar\'e inequalities} \label{sec:pseudo-poincare}

By duality, the previous gradient bounds lead to the following pseudo-Poincar\'e type inequalities:

\begin{proposition} \label{lem:pseudo-poincare}
Let  $f \in C_0^\infty (\mathbb{M})$.
\begin{itemize}
\item If $1 \le p < 2$, then for every $t  \ge 0$,
\begin{align} \label{pseudo-poincare}
 \left\| f- P_t f \right\|_{p} \leq \sqrt{  \left( 2+\frac{4\kappa}{\rho_2} \right) t  }   \| \sqrt{ \Gamma(f) } \| _{p}
\end{align}

\item If $2 \le p \le +\infty$,  then for every $t \ge 0$,
\begin{align} \label{pseudo-poincare2}
 \left\| f- P_t f \right\|_{p} \leq\frac{ \left( 1+\frac{3\kappa}{2\rho_2}\right) \sqrt{2d} }{ \sqrt{ 1+(p-1) \left(1+\frac{3\kappa}{2\rho_2}\right)}} \sqrt{t}  \| \sqrt{ \Gamma(f) } \| _{p}
\end{align}\end{itemize}
\end{proposition}

\begin{proof}
Let  $p' = \frac{p}{p-1}$. For any $g \in C_0^\infty (\M)$ with $\| g \| _{p'} \leq 1 $, we have
\begin{align*}
\int_\M g(f - P_t f ) d\mu &=  \int_\M g \left( -\int_0^t \partial_s P_s f ds \right) d\mu \\
=& -\int_0^t \int_\M g L P_s f d\mu ds = -\int_0^t \int_\M g  P_s L f d\mu ds \\
=& -\int_0^t \int_\M P_s g L f d\mu ds = \int_0^t \int_\M \Gamma( P_s g , f ) d\mu ds  \\
\leq & \| \sqrt{ \Gamma(f) } \|_p \int_0^t \| \sqrt{ \Gamma( P_s g ) } \|_{p'} ds.
\end{align*}

By using   Proposition \ref{GB1}, we have
\[
\int_0^t \| \sqrt{ \Gamma( P_s g ) } \|_{p'} ds\le\int_0^t  \frac{C_{p'}}{\sqrt{s}} ds \| g \|_{p'}.
\]
We therefore obtain
\begin{align*}
\int_\M g(f - P_t f ) d\mu \le 2 C_{p'} \sqrt{t} \|   \sqrt{ \Gamma(f) } \|_p  \| g \|_{p'}.
\end{align*}
By duality we can now conclude that
\[
 \| f- P_t f \|_{p} \le 2 C_{p'} \sqrt{t} \|   \sqrt{ \Gamma(f) } \|_p.
 \]
\end{proof}

\subsection{Improved Sobolev embedding} \label{sec:improved_sobolev}

For $\alpha < 0$, we define the Besov norm $  \| \cdot \|_{B_{\infty,\infty}^{\alpha}} $ on $\mathbb{M}$ as follows :
\begin{align} \label{Besov_norm}
  \| f \|_{B_{\infty,\infty}^{\alpha}} = \sup_{t>0} t^{-\alpha /2} \| P_t f \|_\infty.
\end{align}

It is clear from this definition that  $ \| f \|_{B_{\infty,\infty}^{\alpha}} \leq 1 $ is equivalent to the fact that for every $u >0$,  $ |P_{t_u} f | \leq u$ where $t_u= u^{2/\alpha}$ . For $p \ge 1$, we define then the Sobolev space $W^{1,p}(\mathbb{M})$ as the closure of $C_0^\infty(\M)$ with respect to the norm $\| f \|_p +\|\sqrt{\Gamma(f)}\|_p$.

\begin{theorem}[Improved Sobolev embedding]\label{main}
For every $1 \leq p < q < \infty$ and every $ f \in W^{1,p}(\mathbb{M}) $, we have
\begin{align} \label{main_ineq}
 \| f \|_q \leq C \| \sqrt{ \Gamma(f) } \|_{p}^{\theta} \| f \|_{ B_{\infty ,\infty}^{\theta / (\theta-1)} }^{1-\theta}
\end{align}
 where $\theta = \frac{p}{q}$ and where $C>0$ is a constant that only depends on $p,q,\rho_2,\kappa,d$.
\end{theorem}

\begin{proof}
Techniques of the proof are mainly based on \cite{Led}; for the sake of completeness, we reproduce the main arguments and make sure they adapt to our sub-Riemannian framework. The proof proceeds in three steps.

\textbf{ Step 1.} We  first prove the weak-type inequality
\[
 \| f \|_{q,\infty} \leq C \| \sqrt{\Gamma(f)} \|_p^{\theta} \| f \|_{ B_{\infty ,\infty}^{\theta / (\theta-1)} }^{1-\theta}.
\]

Without loss of generality,  we can assume $\| f \|_{ B_{\infty ,\infty}^{\theta / (\theta-1)} } \leq 1 $, which is equivalent to the condition:
\begin{align} \label{condition}
  |P_{t_u} f | \leq u \textmd{   , }t_u= u^{2(\theta-1)/\theta} \textmd{ for every } u >0 .
\end{align}
We have then
\[
 u^q \mu \{ |f| > 2u \} ~~\leq ~~ u^q \mu \{ |f-P_{t_u} f| > u \} ~~\leq~~ u^{q-p} \int_M |f-P_{t_u} f |^p d\mu
\]
From Proposition \ref{lem:pseudo-poincare}, we have
\[
 \| f- P_t f \|_{p} \leq C \sqrt{t} \| \sqrt{ \Gamma(f) } \| _{p}.
\]
Since $q-p+\frac{p}{2} \frac{2(\theta -1 )}{\theta} = 0$, we conclude
\begin{align*}
 u^q \mu \{ |f| > 2u \} \leq & u^{q-p} \left(C^p t_u^{p/2} \| \sqrt{ \Gamma(f) } \|_p^p  \right) \\
 \leq & C^p  \| \sqrt{ \Gamma(f) } \|_p^p
\end{align*}
We finally observe that $\sup_{u>0} u^q \mu \{ |f| > 2u \} = \frac{1}{2^q} \| f \|_{q,\infty}^q$, to conclude Step 1.

\textbf{ Step 2.} In the previous weak type inequality,  we would like to replace the $L^{q,\infty}$-norm by the $L^q$-norm. Again, we assume $\| f \|_{ B_{\infty ,\infty}^{\theta / (\theta-1)} } \leq 1$, that is
$|P_{t_u} f | \leq u $ for $t_u= u^{2(\theta-1)/\theta} $, $\forall u >0 $.
For $f \in W^{1,p}(\M)  \cap L^q(\M) $ such that $|P_{t_u} f | \leq u $, $\forall u >0$, we want to show that for some constant $C>0$,
\[
 \int_M | f |^q d\mu \leq C \int_\M  \Gamma(f)^{p/2} d\mu.
\]
Let $c\geq 5$ be an arbitrary constant. For any $u>0$, we introduce the truncation
\[
\tilde{f_u} = (f-u)^+ \wedge ((c-1)u) + (f+u)^- \vee ( -(c-1)u).
\]
That is, $\tilde{f_u}(x) = f(x)-u$ when $u\leq f(x) \leq cu$, and $\tilde{f_u}(x)= f(x)+u$
when $-cu\leq f(x) \leq -u$, otherwise $| \tilde{f_u} |$ is truncated as constants $0$ or $(c-1)u$.
Observing $$\{ |f| \geq 5u \} \subset \{ |\tilde{f_u}| \geq 4u \}, $$ yields
\begin{align*}
\int_0^\infty \mu(\{ |f| \geq 5u \} ) & d(u^q) \leq \int_0^\infty \mu(\{ |\tilde{f_u}| \geq 4u \} ) d(u^q) \\
\leq & \int_0^\infty \mu(\{ | \tilde{f_u} -P_{t_u}f  | \geq 3u  \}) d(u^q) \quad (\textmd{ since } | P_{t_u}(f) | \leq u )\\
\leq & \int_0^\infty \mu(\{ | \tilde{f_u} -P_{t_u}\tilde{f_u} |\geq u \}) d(u^q)
 +\int_0^\infty \mu(\{ P_{t_u}(|f-\tilde{f_u} |) \geq 2u \}) d(u^q) .
\end{align*}

We  now apply the pseudo-Poincar\'{e} inequality for $\tilde{f_u}$ as follows,
\begin{align*}
\mu(\{ | \tilde{f_u} -P_{t_u}\tilde{f_u} |\geq u \}) \leq &  u^{-p} \int_\bM  | \tilde{f_u} -P_{t_u}\tilde{f_u} |^p d\mu \\
\leq &  C' u^{-p} t^{p/2}_u \int_\bM  \Gamma(\tilde{f_u})^{p/2} d\mu \\
=& C' u^{-q} \int_{\{u\leq |f| \leq c u \}} \Gamma(f) ^{p/2} d\mu .
\end{align*}

So by integration we get,
\begin{align*}
\int_0^\infty \mu(\{ | \tilde{f_u} -P_{t_u}\tilde{f_u} |\geq u \}) d(u^q) \leq  &
 \int_0^\infty C' q u^{-1}  \int_{\{u\leq |f| \leq c u \}} \Gamma(f)^{p/2} d\mu d u \\
 \leq & C' q \int_\M \Gamma(f)^{p/2}  ~\int_{|f|/c}^{|f|} \frac{d u}{u} ~ d\mu \\
 =& C' q \ln c \int_\M  \Gamma(f)^{p/2} d\mu .
\end{align*}
On the other hand, we have
\begin{align*}
|f-\tilde{f_u}| =& |f-\tilde{f_u}| ~ 1_{\{|f|\leq cu\}} + |f-\tilde{f_u}|~ 1_{\{|f|> cu\}} \\
= \min & (u,|f|)~ 1_{\{|f|\leq cu\}} + (|f| - (c-1)u) ~ 1_{\{|f|> cu\}}  \leq u + |f|~ 1_{\{|f|> cu\}} .
\end{align*}
By integrating, we obtain then
\begin{align*}
\int_0^\infty \mu(\{ P_{t_u}(|f-\tilde{f_u} |) & \geq 2u \}) d(u^q) \leq \int_0^\infty \mu(\{ P_{t_u}(|f|~ 1_{\{|f|> cu\}}) \geq u \}) d(u^q) \\
\leq & \int_0^\infty \frac{1}{u} \left( \int_\M (|f|~ 1_{\{|f|> cu\}}) d\mu \right) d(u^q)
\quad ( P_t \textmd{ is a contraction on }L^1(\mathbb{M}) ) \\
=& \frac{q}{q-1} \int_\M |f| \left( \int_0^\infty 1_{\{|f|> cu\}} d(u^{q-1}) \right)  d\mu \\
=& \frac{q}{q-1} ~~ \frac{1}{c^{q-1}} \| f \|_q^q .
\end{align*}
Gathering all the estimates, we can then conclude
\begin{align*}
 \frac{1}{5^q} \int_\M | f |^q d\mu  =& \frac{1}{5^q} \|f\|_q^q = \int_0^\infty \mu(\{ |f| \geq 5u \} )  d(u^q) \\
\leq& C' q \ln c \int_\M \Gamma(f)^{p/2} d\mu + \frac{q}{q-1} ~~ \frac{1}{c^{q-1}} \| f \|_q^q
\end{align*}
If we pick a large $c \geq 5$ depending on $q$ such that $\frac{1}{5^q} > \frac{q}{q-1} ~~ \frac{1}{c^{q-1}} $,  we have proved the claim $$ \| f \|_q^q  \leq C^q \| \sqrt{ \Gamma(f) } \|^p $$ with $C = \left( \frac{C' q \ln c}{\frac{1}{5^q} -\frac{q}{(q-1)c^{q-1}} } \right)^{1/q}$.

\textbf{ Step 3.} Finally, it remains to prove $\|f\|_q < \infty$ is actually a consequence of  $ \|\sqrt{\Gamma(f)}\|_p<\infty $,
$\| f \|_{ B_{\infty ,\infty}^{\theta / (\theta-1)} } \leq 1 $, so that we can remove the  condition $f \in L^q(\mathbb{M}) $ from Step 2 and
complete the proof of  theorem. From the weak type inequality of Step 1, we have $\| f\|_{q,\infty} <\infty $. For any $0<\epsilon<1 $, we define
\[
N_\epsilon (f) = \int_{\epsilon}^{1/\epsilon} \mu( \{ |f| \geq 5u \} ) d(u^q)
\leq  \frac{2q}{5^q} \left(\ln\frac{1}{\epsilon} \right) \| f\|_{q,\infty}^q < \infty.
\]
Following the argument in Step 2 again, we see that
\[
N_\epsilon (f) \leq C' q \ln c \int_\M \Gamma(f)^{p/2} d\mu +
\int_{\epsilon}^{1/\epsilon} \frac{1}{u} \left( \int_\M (|f|~ 1_{\{|f|> cu\}}) d\mu \right) d(u^q) .
\]
The first term is bounded, and the second term can be  estimated as follows.
\begin{align*}
 & \int_{\epsilon}^{1/\epsilon} \frac{1}{u}  \left(  \int_\M (|f|~ 1_{\{|f|> cu\}}) d\mu  \right) d(u^q) \\
 = & \int_{\epsilon}^{1/\epsilon} \frac{1}{u}  \left(  c u \mu( {\{ |f|>cu\}} ) + c\int_u^\infty \mu( {\{|f|> cv \}} ) dv  \right) d(u^q) \\
\leq & (c+ \frac{c}{q-1}) \int_{\epsilon}^{1/\epsilon} \mu(\{|f| \geq cu \}) d(u^q) +
 \frac{cq}{(q-1)\epsilon^{q-1}} \int_{1/\epsilon}^{\infty} \mu(\{|f| \geq cu \}) du \\
\leq & \frac{q}{q-1}\frac{5^q}{c^{q-1}} N_\epsilon (f) + \frac{cq}{q-1} \int_{5/c\epsilon}^{1/\epsilon} \frac{ \| f\|_{q,\infty}^q }{(cu)^q}  d(u^q) +
 \frac{cq}{(q-1)\epsilon^{q-1}} \int_{1/\epsilon}^{\infty}  \frac{ \| f\|_{q,\infty}^q }{(cu)^q} du \\
 = & \frac{q}{q-1}\frac{5^q}{c^{q-1}} N_\epsilon (f)   + \frac{q}{q-1} \frac{1}{c^{q-1}}
 \| f\|_{q,\infty}^q \left( q \ln\frac{c}{5}+ \frac{1}{q-1} \right)
\end{align*}
So, by  choosing $c$ large enough, we have $\sup_{0<\epsilon<1} N_\epsilon (f) <\infty $ which implies
$\|f\|_q  = \lim_{\epsilon \rightarrow 0 } 5 (N_\epsilon (f) )^{1/q}    < \infty$.  This completes the proof.

\end{proof}

\subsection{Sobolev inequality,  Isoperimetry and volume growth}

In this section, we study the Sobolev and isoperimetric inequalities and their connections with the volume growth of metric balls. We obtain the sub-Riemannian analogue of a theorem essentially  due to Ledoux \cite{ledoux-stflour}.

 We first  remind what we mean by the perimeter of a  set  in our subelliptic setting. For further details, we refer to \cite{GN2}.

Let us first observe that, given any point $x\in \M$ there exists an open set $x\in U\subset \M$ in which the operator $L$ can be written as
\begin{equation}\label{locrep}
L = - \sum_{i=1}^m X^*_i X_i,
\end{equation}
where  the vector fields $X_i$ have Lipschitz continuous coefficients in $U$, and $X_i^*$ indicates the formal adjoint of $X_i$ in $L^2(\M,d\mu)$.

We indicate with $\mathcal F(\bM )$ the set of $C^1$ vector fields which are subunit for $L$. Given a function $f\in
L^1_{loc}(\bM)$, which is supported in $U$ we define the horizontal total variation of $f$ as
\[
\text{Var}(f) = \underset{\phi\in \F(\bM)}{\sup} \int_U
f \left(\sum_{i=1}^m X^*_i \phi_i\right) d\mu,
\]
where on $U$, $\phi=\sum_{i=1}^m \phi_i X_i$. For functions not supported in $U$, $\text{Var}(f) $ may be defined by using a partition of unity.
The space \[ BV (\bM) = \{f\in L^1(\bM)\mid
\text{Var}(f)<\infty\},
\]
endowed with the norm
\[
||f||_{BV(\bM)} = ||f||_{L^1(\bM)} + \text{Var}(f),
\]
is a Banach space. It is well-known that $W^{1,1}(\bM) = \{f\in
L^1(\bM)\mid \sqrt{\Gamma f}\in L^1(\bM)\}$ is a strict subspace of
$BV(\bM)$ and when $f\in
W^{1,1}(\bM)$ one has in fact
\[
\text{Var}(f) = ||\sqrt{\Gamma(f)}||_{L^1(\bM)}.
\]
Given a measurable set $E\subset \bM$ we say that it has finite perimeter, or is a Cacciopoli set if $\mathbf 1_E\in BV(\bM)$. In
such case the perimeter of $E$ is by
definition
\[
P(E) = \text{Var}(\mathbf 1_E).
\]
In a later section, we will need the following approximation result, see Theorem 1.14 in \cite{GN2}.

\begin{lemma}\label{P:ag}
Let $f\in BV(\bM)$, then there exists a sequence $\{f_n\}_{n\in
\mathbb N}$ of functions in $C_0^\infty(\bM)$ such that:
\begin{itemize}
\item[(i)] $||f_n - f||_{L^1(\bM)} \to 0$;
\item[(ii)] $\int_\bM \sqrt{\Gamma(f_n)} d\mu \to
\text{Var}(f)$.
\end{itemize}
\end{lemma}

We now prove the main result of this subsection.

\begin{theorem}\label{P:char}
Let $D>1$. Let us assume that $\bM$ is not compact in the metric
topology, then the following assertions  are equivalent:
\begin{itemize}
\item[(1)] There exists a constant $C_1 >0$ such that for every $x \in \bM$, $r \ge 0$,
\[
\mu (B(x,r)) \ge C_1 r^D.
\]
\item[(2)] There exists a constant $C_2>0$ such that for $x \in \bM$, $t>0$,
\[
p(x,x,t) \le \frac{C_2}{t^{\frac{D}{2}}}.
\]
\item[(3)] For some $1\leq p,q,r < \infty$ with $ \frac{1}{q}=\frac{1}{p}-\frac{r}{qD} $, there exists a constant $C_3>0$ such that $\forall f \in C_0^\infty(\mathbb{M})$, we have
\[
\| f \|_q \leq C_3 \| \sqrt{\Gamma(f)} \|_p^{p/q} ~\|f\|_r^{1-p/q}.
\]

\item[(4)] There exists a constant $C_4 >0$ such that for every Caccioppoli set $E\subset \bM$ one has
\[
\mu(E)^{\frac{D-1}{D}} \le C_4 P(E).
\]
\end{itemize}
\end{theorem}
\begin{remark}  if we replace the condition of (3) by for all $1\leq p,q,r < \infty$ with $ \frac{1}{q}=\frac{1}{p}-\frac{r}{qD} $,   (1),(2), (3) and (4) would  still  be equivalent.
\end{remark}
\begin{proof}
That (1) $\rightarrow$ (2) follows immediately from the Li-Yau Gaussian upper bound
\[
p(x,x,t) \le \frac{C}{\mu(B(x,\sqrt{t})}
\]
 that is proved in \cite{Bau2}.

The proof that (2) $\rightarrow$ (3)  follows from the improved Sobolev embedding Theorem \ref{main}.

Indeed, (2)  implies first  that for $x,y \in \M$,
\begin{align*}
p(x,y,t)& =\int_\M p(x,z,t/2) p( z,y,t/2) \mu(dy) \\
 &  \le \sqrt{\int_\M p(x,z,t/2)^2 \mu(dz)  } \sqrt{\int_\M p(y,z,t/2)^2 \mu(dz)  } \\
 & = \sqrt{p(x,x,t)p(y,y,t)} \\
 &  \le \frac{C_2}{t^{\frac{D}{2}}}.
\end{align*}

Therefore,  for every $f \in L^1(\mathbb{M})$, we have
\begin{align*}
 \| P_t(f)\|_\infty = \left\| \int_\M p(\cdot,y,t) f(y) \mu(dy) \right\|_\infty  \le \|p(\cdot,y,t)\|_\infty \|f\|_1 \leq \frac{C_2}{ t^{D/2} } \|f\|_1 .
\end{align*}
On the other hand, $P_t$ is a contraction on $L^\infty(\mathbb{M})$, i.e. $\|P_t\|_{\infty\rightarrow\infty} \leq 1$. Therefore, by the  Riesz-Thorin interpolation theorem, we deduce that we have  the following heat semigroup embedding
\begin{align*}
 \| P_t\|_{r\rightarrow\infty} \leq \frac{C_2^{1/r}}{t^{D/2r}} , ~~ r \geq 1 .
\end{align*}
Let now  $1\leq p,q,r < \infty$ such that$ \frac{1}{q}=\frac{1}{p}-\frac{r}{qD} $. Since  for $\theta=\frac{p}{q}$, $-\frac{\theta}{2(\theta-1)} - \frac{D}{2r} =0$, we have
\begin{align*}
\|f\|_{B_{\infty,\infty}^{\theta/(\theta-1)}} =& \sup_{t>0} t^{-\theta/2(\theta-1)} \|P_t f\|_\infty \\
\leq & \sup_{t>0} t^{-\theta/2(\theta-1)} \frac{C_2^{1/r}}{t^{D/2r}} \|f\|_r = C_2^{1/r} \|f\|_r ,
\end{align*}
we can conclude  (3) from the improved Sobolev embeddding of Theorem \ref{main}.

The proof that (3) is equivalent to (4) follows the classical ideas
of Fleming-Rishel and Maz'ya, and it is based on a generalization of
Federer's co-area formula for the space $BV(\M)$, see for instance
\cite{GN2}.

Finally, we show that $(3) \rightarrow (1)$. We adapt an idea in
\cite{SC} (see Theorem 3.1.5 on p. 58).
For any fix $x\in \mathbb{M}$, $s>0$, consider the function
\[
 f(y) = \max\{ s-d(x,y),0 \}.
\]
Then, it is easily seen that
\begin{align*}
 \| f \|_q \geq & (s/2) \mu(B(x,s/2) )^{1/q} \\
 \| f \|_r \leq & s \mu(B(x,s))^{1/r} \\
 \| \sqrt{\Gamma(f)} \|_p \leq &  \mu(B(x,s))^{1/p} .
\end{align*}
Hence, from (3) we have
\[
 \mu(B(x,s/2) )^{1/q} \leq 2 C_3 s^{-p/q} \mu(B(x,s))^{1/q+(1/r)(1-p/q)} =2 C_3 s^{-p/q} \mu(B(x,s))^{1/q+p/qD}.
\]
This can be written as follows.
\[
  \mu(B(x,s)) \geq  (2C_3)^{-Dq/(D+p)}  s^{Dp/(D+p)} \mu(B(x,s/2) )^{D/(D+p)}.
\]
\[
  \mu(B(x,s)) \geq  \{ (2C_3)^{-q}  s^{p} \}^a \mu(B(x,s/2) )^{a}
\]
where $a=D/(D+p) <1$.
Replacing $s$ by $s/2$ iteratively, we obtain
\[
  \mu(B(x,s)) \geq  (2C_3)^{-q (\sum_{j=1}^{i} a^j)}  s^{p (\sum_{j=1}^{i} a^j)} 2^{-p (\sum_{j=1}^{i} (j-1)a^j)} \mu(B(x,s/2^i) )^{a^i}.
\]

From the volume doubling property proved  in \cite{BBG}, we have the   control $$\mu(B(x,s/2^i) ) \geq C^{-1} (1/2^i)^Q \mu(B(x,s)), $$ for some $C = C(\rho_1,\rho_2,\kappa,d)>0$ and
$Q= \log_2 C$.

Therefore, we have
\[
 \liminf_{i\rightarrow \infty} \mu(B(x,s/2^i) )^{a^i} \geq  \lim_{i\rightarrow \infty} ( C^{-1}  \mu(B(x,s)) )^{a^i}  (1/2)^{iQ a^i } = 1.
\]
Since $\sum_{j=1}^\infty a^j = D/p$, $\sum_{j=1}^\infty (j-1)a^j = D^2/p^2$, we obtain the volume growth control
\[
  \mu(B(x,s)) \geq 2^{-(q+D)D/p} C_3^{-qD/p} s^{D}.
\]
This establishes (1), thus completing the proof.

\end{proof}

\begin{remark}
By combining the results of \cite{BBG} and   \cite{GN2},  an alternative proof of (1) $\rightarrow$ (4) could be given.  Indeed,
in \cite{GN2} it was proved that in a Carnot-Carath\'eodory space
$(X,\mu,d)$ the doubling condition \[ \mu(B(x,2r)) \le C_1
\mu(B(x,r)), \ \ \ x\in X, r>0, \] for the volume of the metric
balls combined with a weak Poincar\'e inequality suffice to
establish the following basic relative isoperimetric inequality
\begin{align}\label{iso}
\min\left\{\mu(E\cap B(x,r)),\mu((X\setminus E)\cap
B(x,r))\right\}^{\frac{D-1}{D}} \le C_{\text{iso}}
\left(\frac{r^D}{\mu(B(x,r))}\right)^{\frac{1}{D}} P(E,B(x,r)),
\end{align}
where $E\subset X$ is any set of
locally finite perimeter. In this inequality the number $D = \log_2
C_1$, where $C_1$ is the doubling constant, and $C_{\text{iso}}$ is
a constant which depends only on $C_1$ and on the constant in the
Poincar\'e inequality. If in addition the space $X$ satisfies the
 volume growth condition
\begin{equation}\label{maxvolgrowth}
\mu(B(x,r))\ge C_2 r^D,\ \ \ x\in \bM, r>0,
\end{equation}
then \eqref{iso} gives the global isoperimetric inequality
\begin{equation}\label{isoglobal}
\mu(E)^{\frac{D-1}{D}} \le C_{\text{iso}} P(E,\bM),
\end{equation}
for any measurable set of locally finite perimeter $E\subset \bM$. Since in \cite{BBG}, it was proved that the doubling condition and the weak Poincar\'e inequality are satisfied when $\rho_1\ge0$, we conclude that  (1) $\rightarrow$ (4).
\end{remark}

\section{The case $\rho_1 >0$}

Throughout this Section 3, we assume that $L$ satisfies the  \emph{generalized curvature-dimension inequality} \emph{CD}$(\rho_1,\rho_2,\kappa,d)$ with  $\rho_1>0$,  $\rho_2 >0$ and $\kappa \ge 0$.
The  following gradient bound was also proved in \cite{Bau2}.
\begin{theorem}[Li-Yau type gradient estimate with $\rho_1>0$]\label{T:ge2}
Let  $f \in C_0^\infty(\M)$, $f  \ge 0$, $f \not\equiv 0$, then the following inequality holds for $t>0$:
\begin{align}\label{bla}
\Gamma (\ln P_t f) \le
\frac{2\rho_2+3\kappa}{2\rho_2} e^{-\frac{2\rho_1\rho_2}{3(\rho_2+\kappa)} t} \frac{LP_t
f}{P_t f}+ \frac{d \rho_1}{12 \rho_2} \frac{(2\rho_2+
3\kappa)^2}{\rho_2+\kappa}  \frac{e^{-\frac{4\rho_1\rho_2}{3(\rho_2+\kappa)} t}}{
1-e^{-\frac{2\rho_1\rho_2}{3(\rho_2+\kappa)} t}}.
\end{align}
\end{theorem}

\subsection{Gradient bounds for the heat semigroup}

We first establish the following reverse Poincar\'e inequality.

\begin{proposition} \label{reverse_poincare2}
For $f \in C^\infty_0(\M)$ and $t \ge 0$,
\[
 \Gamma(P_t f)  \leq \frac{1}{2} \rho_1 \frac{ \rho_2+2\kappa}{ \rho_2+\kappa}  \frac{ e^{-2\frac{\rho_1\rho_2 }{\rho_2+\kappa}t}}{1- e^{-\frac{\rho_1\rho_2}{\rho_2+\kappa} t} } (P_t f^2 -(P_t f)^2).
\]
\end{proposition}

\begin{proof}
Let us fix $T>0$ once time for all in the following proof. Given a function  $f \in C^\infty_0(\M)$, for $0\le t\le T$ we introduce the functionals

\[
\phi_1 (x,t)=\Gamma ( P_{T-t}f)(x),
\]
and
\[
\phi_2 (x,t)=  \Gamma^Z (P_{T-t}f)(x),
\]
which are defined on $\M\times [0,T]$.  A straightforward computation shows  that
\[
L\phi_1+\frac{\partial \phi_1}{\partial t} =2 \Gamma_2 (P_{T-t}f).
\]
and
\[
L\phi_2+\frac{\partial \phi_2}{\partial t} =2 \Gamma_2^Z (P_{T-t}f).
\]

Consider now the function
\begin{align*}
\phi (x,t)&= a(t) \phi_1 (x,t)+b(t) \phi_2(x,t) \\
 & =a(t)\Gamma ( P_{T-t}f)(x)+b(t) \Gamma^Z ( P_{T-t}f)(x),
\end{align*}
where $a$ and $b$ are two non negative functions  that will be chosen later.
Applying the generalized curvature-dimension inequality \emph{CD}$(\rho_1,\rho_2,\kappa,\infty)$, we obtain
\begin{align*}
  L\phi+\frac{\partial \phi}{\partial t} =&
a'  \Gamma (P_{T-t}f)+b'  \Gamma^Z ( P_{T-t}f)  \\
 &+2a \Gamma_2 ( P_{T-t}f)+2b\Gamma_2^Z (P_{T-t}f) \\
&\ge  \left(a'+2\rho_1 a -2\kappa \frac{a^2}{b}\right) \Gamma ( P_{T-t}f)  +(b'+2\rho_2 a) \Gamma^Z ( P_{T-t}f).
\end{align*}
Let us now chose
\[
b(t)=\left( e^{-\frac{2\rho_1 \rho_2 t}{\kappa+\rho_2}}- e^{-\frac{2\rho_1 \rho_2 T}{\kappa+\rho_2}}\right)^2
\]
and
\[
a(t)=-\frac{b'(t)}{2\rho_2},
\]
so that
\begin{align*}
b'+2\rho_2 a=0
\end{align*}
and
\begin{align*}
a'+2\rho_1 a -2\kappa \frac{a^2}{b}  \ge \rho_1 \frac{ \rho_2+2\kappa}{ \rho_2+\kappa}   e^{-2\frac{\rho_1\rho_2 }{\rho_2+\kappa}T} .
\end{align*}
With this choice, we get therefore
\[
 L\phi+\frac{\partial \phi}{\partial t}  \ge- \rho_1 \frac{ \rho_2+2\kappa}{ \rho_2+\kappa}   e^{-2\frac{\rho_1\rho_2 }{\rho_2+\kappa}T} \Gamma ( P_{T-t}f).
\]
and therefore from a comparison theorem for parabolic partial differential equations (see \cite{Bau2})  we have
\[
P_T(\phi(\cdot,T))(x) \ge \phi(x,0)-\rho_1 \frac{ \rho_2+2\kappa}{ \rho_2+\kappa}   e^{-2\frac{\rho_1\rho_2 }{\rho_2+\kappa}T}\int_0^T P_t (\Gamma ( P_{T-t}f))dt .
\]
It is easily seen that
\[
\int_0^T P_t (\Gamma ( P_{T-t}f))dt=\frac{1}{2} ( P_Tf^2-(P_T f)^2),
\]
and since,
\[
\phi(x,0)=a(0)\Gamma ( P_{T}f)(x)+b(0)\Gamma^Z ( P_{T}f)(x)
\]
and
\[
P_T(\phi(\cdot,T))(x) =a(T)P_{T}( \Gamma (f))(x)+b(T)P_{T}(  \Gamma^Z ( f))(x)=0,
\]
the proof is completed.
\end{proof}

\begin{proposition} \label{GB2}
Let  $f \in C_0^\infty(\bM)$.
\begin{itemize}
\item If $1 \le p < 2$, then for every $t >0$,
\[
\left\| \sqrt{\Gamma(P_t f)} \right\|_p  \leq
 \frac{ 1 }  {  \left(  1+(p-1)(1+\frac{3\kappa}{2\rho_2} )e^{-\frac{2\rho_1\rho_2}{3(\rho_2+\kappa)}t} \right)^{\frac{1}{2} }  }
 \left(   \frac{ d\rho_1\rho_2 } { 3(\rho_2+\kappa) }
          \frac{ ( 1+\frac{3\kappa}{2\rho_2} )^2 e^{-\frac{4\rho_1\rho_2}{3(\rho_2+\kappa)}t} }  { (1-e^{- \frac{2\rho_1\rho_2} {3(\rho_2+\kappa)} t} ) } \right)^\frac 1 2
  \|f\|_p
\]

\item If $2 \le p \le +\infty$,  then for every $t>0$,
\[
\left\| \sqrt{\Gamma(P_t f)} \right\|_p  \leq  \left( \frac{1}{2} \rho_1 \frac{ \rho_2+2\kappa}{ \rho_2+\kappa}  \frac{ e^{-2\frac{\rho_1\rho_2 }{\rho_2+\kappa}t}}{1- e^{-\frac{\rho_1\rho_2}{\rho_2+\kappa} t} } \right)^\frac{1}{2} \| f\|_p.
\]
\end{itemize}
\end{proposition}

\begin{proof}
The proof is essentially identical to the proof of Proposition \ref{GB1}. We observe from this proof that if for $f\in C^\infty_0 (\M),t>0 $
\begin{align*}
 &\Gamma( \ln P_t f) \leq \alpha(t) \frac{LP_t f}{P_t f} + \beta(t), \quad \quad f\geq 0,f\not\equiv 0, \alpha(t),\beta(t) >0\\
 &\Gamma( P_t f) \leq \gamma(t) ( P_t f^2 - (P_t f)^2) ,\quad \quad \gamma(t) >0
\end{align*}
then
\begin{align*}
 \| \sqrt{\Gamma(P_t f)} \|_p \leq  \left( \frac{\beta(t)}{ 1+ (p-1) \alpha(t) } \right)^\frac{1}{2} \| f\|_p,  \ \ \ &\text{for }1\leq p < 2   \\
 \| \sqrt{\Gamma(P_t f)} \|_p \leq  (\gamma(t))^\frac{1}{2} \| f\|_p,  \quad \quad \quad\quad \ &\text{for }2\leq p < \infty .
\end{align*}
By Theorem \ref{T:ge2} and Proposition \ref{reverse_poincare2}, we then see that  $ \alpha(t),\beta(t),\gamma(t)$ are  given by:
\begin{align*}
 \alpha(t) &=   (1+\frac{3\kappa}{2\rho_2} )e^{-\frac{2\rho_1\rho_2}{3(\rho_2+\kappa)}t}   , \quad
 \beta(t) =  \frac{ d\rho_1\rho_2 } { 3(\rho_2+\kappa) }
          \frac{ ( 1+\frac{3\kappa}{2\rho_2} )^2 e^{-\frac{4\rho_1\rho_2}{3(\rho_2+\kappa)}t} }  { (1-e^{- \frac{2\rho_1\rho_2} {3(\rho_2+\kappa)} t} ) }    \\
 &\gamma(t) = \frac{1}{2} \rho_1 \frac{ \rho_2+2\kappa}{ \rho_2+\kappa}  \frac{ e^{-2\frac{\rho_1\rho_2 }{\rho_2+\kappa}t}}{1- e^{-\frac{\rho_1\rho_2}{\rho_2+\kappa} t} } .
\end{align*}
\end{proof}

\subsection{Pseudo-Poincar\'e inequalities} \label{sec:pseudo-poincare2}

\begin{proposition} \label{lem:pseudo-poincare2}
Let  $f \in C_0^\infty (\mathbb{M})$.
\begin{itemize}
\item If $1 \le p < 2$, then for every $t  \ge 0$,
\begin{align} \label{pseudo-poincare3}
 \| f- P_t f \|_{p} \leq \left( \frac{2(\rho_2+2\kappa)(\rho_2+\kappa)}{\rho_1 \rho_2^2} ( 1- e^{-\frac{\rho_1\rho_2}{\rho_2+\kappa} t} ) \right)^\frac{1}{2}
   \| \sqrt{ \Gamma(f) } \| _{p}
\end{align}

\item If $2 \le p \le +\infty$,  then for every $t \ge 0$,
\begin{align} \label{pseudo-poincare4}
 \| f- P_t f \|_{p} \leq ( 1+\frac{3\kappa}{2\rho_2} ) \left( \frac{3d(\rho_2 + \kappa)}{\rho_1\rho_2}  (1-e^{- \frac{2\rho_1\rho_2} {3(\rho_2+\kappa)} t} ) \right)^\frac{1}{2}
   \| \sqrt{ \Gamma(f) } \| _{p}
\end{align}\end{itemize}
\end{proposition}

\begin{proof}
As shown in the proof of Proposition \ref{lem:pseudo-poincare}, we have
\begin{align*}
 \| f - P_t f \|_p \leq \left( \int_0^t \sqrt{\gamma(s)} ds \right) \| \sqrt{\Gamma(f)} \|_p , \ \ \ \ \text{for } 1\leq p < 2 \\
 \| f - P_t f \|_p \leq \left( \int_0^t \sqrt{\frac{\beta(s)}{1+(p-1)\alpha(s)}} ds \right) \| \sqrt{\Gamma(f)} \|_p
      , \ \ \ \ \text{for } 2\leq p < \infty \\
\end{align*}
where $\alpha,\beta,\gamma$ are defined in the proof of Proposition \ref{GB2}. The proof is done by
\begin{align*}
 \int_0^t \sqrt{\gamma(s)} ds &= \int_0^t \left( \frac{\rho_1(\rho_2+2\kappa)}{2(\rho_2+\kappa)} \right)^\frac{1}{2}
  \frac{ e^{-\frac{\rho_1\rho_2 }{\rho_2+\kappa}s}}{\sqrt{ 1- e^{-\frac{\rho_1\rho_2}{\rho_2+\kappa} s} }} ds \\
  &= \left( \frac{2(\rho_2+2\kappa)(\rho_2+\kappa)}{\rho_1 \rho_2^2} ( 1- e^{-\frac{\rho_1\rho_2}{\rho_2+\kappa} t} ) \right)^\frac{1}{2}\\
 \int_0^t \sqrt{\frac{\beta(s)}{1+(p-1)\alpha(s)}} & ds  \leq \int_0^t \sqrt{\beta(s)} ds \\
  &= \int_0^t \left( \frac{ d\rho_1\rho_2 } { 3(\rho_2+\kappa) }
          \frac{ ( 1+\frac{3\kappa}{2\rho_2} )^2 e^{-\frac{4\rho_1\rho_2}{3(\rho_2+\kappa)}s} }  { (1-e^{- \frac{2\rho_1\rho_2} {3(\rho_2+\kappa)} s} ) } \right)^\frac{1}{2} ds \\
  &= \left( 1+\frac{3\kappa}{2\rho_2} \right) \left( \frac{3d(\rho_2 + \kappa)}{\rho_1\rho_2}  (1-e^{- \frac{2\rho_1\rho_2} {3(\rho_2+\kappa)} t} ) \right)^\frac{1}{2}
\end{align*}
\end{proof}

\subsection{Poincar\'e inequality}
In the case of $\rho_1 >0$, we have the following  theorem which is proved in \cite{BB}.
\begin{theorem} \label{T:finite_measure}
The measure $\mu$ is finite, i.e. $\mu(\M) < +\infty$ and for every $1 \le p \le \infty$, $f\in L^p(\M)$,
\[
 P_t f \xrightarrow{t\rightarrow \infty} \frac{1}{\mu(\M)} \int_\M f d\mu .
\]
\end{theorem}
This theorem allows us to deduce the Poincar\'{e} inequality.
\begin{proposition} \label{poincare_ineq}
Let $1\leq p <\infty$. There exists $C= C_p(\rho_1,\rho_2,\kappa,d)>0$ such that, for $\forall f\in C_0^\infty (\M)$,
\begin{align*}
 \| f- f_\M \|_p \leq  C  \| \sqrt{\Gamma(f)} \|_p ,
\end{align*}
where $f_\M = \frac{1}{\mu(\M)} \int_\M f d\mu $.
\end{proposition}

\begin{proof}
The proof is immediate from Proposition \ref{lem:pseudo-poincare2} and Theorem \ref{T:finite_measure} by letting $t \to \infty$. And $C$ is given by
\[
C_p(\rho_1,\rho_2,\kappa,d) = \left\{ \begin{array}{ll}
            \left( \frac{2(\rho_2+2\kappa)(\rho_2+\kappa)}{\rho_1 \rho_2^2}  \right)^\frac{1}{2} & \text{if } 1\leq p<2, \\
            ( 1+\frac{3\kappa}{2\rho_2} ) \left( \frac{3d(\rho_2 + \kappa)}{\rho_1\rho_2} \right)^\frac{1}{2} & \text{if } 2\leq p<\infty. \end{array} \right.
\]
\end{proof}

\subsection{A lower bound on the Cheeger's isoperimetric constant}

In \cite{Che}, in order to bound from below the first eigenvalue $\lambda_1$  of a compact Riemannian manifold with normalized Riemannian measure $\mu_g$, Cheeger's introduced   the following isoperimetric constant
\[
h=\inf \frac{\mu_g (\partial A)}{\mu_g(A)},
\]
where the infimum runs over all open subsets $A$ with smooth boundary $\partial A$ such that $\mu(A)\le \frac{1}{2}$. Cheeger's inequality then writes $\lambda_1 \ge \frac{h^2}{4}$.

Such isoperimetric quantity may also be considered and estimated in our sub-Riemannian framework. Throughout this section, we  assume $\mu(\M) = 1$. Let
\[
\iota=\inf \frac{P(E)}{\mu(E)}
\]
where the infimum runs over all Caccioppoli sets $E$ such that $\mu(E)\le \frac{1}{2}$ (we remind that $P(E)$ denotes the perimeter of $E$ as defined in Section 2.4 ). By following the argument of Ledoux in \cite{ledoux2} we see that $\lambda_1 \ge \frac{i^2}{4}$ where $\lambda_1$ is the first eigenvalue of $-L$. The next proposition gives a lower bound on $\iota$ (and therefore on $\lambda_1$).

\begin{proposition}\label{P:isocompact}
Let $E\subset \bM$ be a Caccioppoli  set. We have
\[
\mu(E)(1-\mu(E))  \le  \sqrt{\frac{2}{\rho_1} }  \left( 1+\frac{2\kappa}{\rho_2}  \right) P(E).
\]
As a consequence
\[
\iota \ge \frac{1}{2}\sqrt{\frac{\rho_1}{2} } \frac{1}{  1+\frac{2\kappa}{\rho_2}  }.
\]
\end{proposition}

\begin{proof}
We know from the pseudo-Poincar\'e inequality that for $f \in C^\infty_0(\M)$,
\begin{align}\label{PoincareP_t}
\|P_tf - f\|_{1} \le \sqrt{\frac{2}{\rho_1} }  \left( 1+\frac{2\kappa}{\rho_2}  \right)  \sqrt{1- e^{-\frac{\rho_1\rho_2}{\rho_2+\kappa} t}} \|
\sqrt{\Gamma(f)} \|_{1},\ \ \ t>0.
\end{align}
Suppose now that $E\subset \bM$ is a Caccioppoli set. By Proposition \ref{P:ag} there exists a sequence
$\{f_n\}_{n\in \mathbb N}$ in $C^\infty_0(\bM)$ satisfying (i) and
(ii) of that Proposition. Applying \eqref{PoincareP_t} to $f_n$ we obtain \[ \|P_tf_n -
f_n\|_{1} \le \sqrt{\frac{2}{\rho_1} }  \left( 1+\frac{2\kappa}{\rho_2}  \right) \sqrt{1- e^{-\frac{\rho_1\rho_2}{\rho_2+\kappa} t}}  \| \sqrt{\Gamma(f_n)}
\|_{1} = \sqrt{\frac{2}{\rho_1} }  \left( 1+\frac{2\kappa}{\rho_2}  \right) \sqrt{1- e^{-\frac{\rho_1\rho_2}{\rho_2+\kappa} t}}  \text{Var}(f_n).
\]
Letting $n\to \infty$ in this inequality, we conclude
\[ \|P_t \mathbf 1_E -
\mathbf 1_E\|_{L^1(\bM)} \le \sqrt{\frac{2}{\rho_1} }  \left( 1+\frac{2\kappa}{\rho_2}  \right) \sqrt{1- e^{-\frac{\rho_1\rho_2}{\rho_2+\kappa} t}}
\text{Var}(\mathbf 1_E) = \sqrt{\frac{2}{\rho_1} }  \left( 1+\frac{2\kappa}{\rho_2}  \right) \sqrt{1- e^{-\frac{\rho_1\rho_2}{\rho_2+\kappa} t}}  P(E).
\]
Observe now that, using $P_t 1 = 1$, we have
\begin{align*}
\|P_t \mathbf 1_E - \mathbf 1_E\|_{L^1(\bM)} & \geq \int_\M | \mathbf 1_{E^c} | ~~|P_t \mathbf 1_E - \mathbf 1_E | d\mu \\
\geq & \int_\M \mathbf 1_{E^c} (P_t \mathbf 1_E - \mathbf 1_E ) d\mu =  \int_\M \mathbf 1_{E^c} P_t \mathbf 1_E  d\mu \\
= & \int_\M P_t \mathbf 1_E d\mu - \int_\M \mathbf 1_E  P_t \mathbf 1_E d\mu =  \int_\M \mathbf 1_E d\mu - \int_E P_t \mathbf 1_E d\mu \\
= & \mu(E) -\int_E P_t \mathbf 1_E d\mu
\end{align*}
On the other hand, from the semigroup property we have
\[
\int_E  P_t \mathbf 1_E d\mu  = \int_\bM \left(P_{t/2}\mathbf
1_E\right)^2 d\mu.
\]
We thus obtain
\[
||P_t \mathbf 1_E - \mathbf 1_E||_{L^1(\bM)} \geq  \left(\mu(E) -
\int_\bM \left(P_{t/2}\mathbf 1_E\right)^2 d\mu\right).
\]
In \cite{Bau2}, it has been proved that for $x,y \in \bM$ and $t>0$,
\[
p(x,y,t) \le \frac{1}{\left( 1-e^{-\frac{2\rho_1 \rho_2
t}{3(\rho_2+\kappa)}}
\right)^{\frac{d}{2}\left(1+\frac{3\kappa}{2\rho_2}\right)} }.
\]

This gives
\begin{align*}
\int_\bM (P_{t/2} \mathbf 1_E)^2 d\mu & \le \left(\int_E
\left(\int_\bM p(x,y,t/2)^2
d\mu(y)\right)^{\frac{1}{2}}d\mu(x)\right)^2
\\
& = \left(\int_E p(x,x,t)^{\frac{1}{2}}d\mu(x)\right)^2 \le
\frac{1}{\left( 1-e^{-\frac{2\rho_1 \rho_2
t}{3(\rho_2+\kappa)}}
\right)^{d\left(1+\frac{3\kappa}{2\rho_2}\right)} } \mu(E)^2.
\end{align*}
Combining these equations we reach the conclusion

\[
 \sqrt{\frac{2}{\rho_1} }  \left( 1+\frac{2\kappa}{\rho_2}  \right) \sqrt{1- e^{-\frac{\rho_1\rho_2}{\rho_2+\kappa} t}} P(E) \ge
  \mu(E) -\frac{1}{\left( 1-e^{-\frac{2\rho_1 \rho_2
t}{3(\rho_2+\kappa)}}
\right)^{d\left(1+\frac{3\kappa}{2\rho_2}\right)} } \mu(E)^2 .
\]
We conclude by letting $t \to +\infty$.
\end{proof}

\subsection{A Lichnerowicz type theorem}\label{lichnerowicz}

A well-known theorem of Lichnerowicz asserts that on a
$d$-dimensional complete Riemannian manifold whose Ricci curvature
is bounded below by a non negative constant $\rho$, then the first
eigenvalue of the Laplace-Beltrami operator is bounded below by
$\frac{\rho d}{d-1}$. In this section, we provide a similar theorem
for our operator $L$. Let us observe that in \cite{Greenleaf},
Greenleaf obtained a similar result for the sub-Laplacian on a CR
manifold. A recent work of Hladky \cite{Hla} also gives lower bounds for the first eigenvalue of sub-Laplacians on some sub-Riemannian manifolds.

\begin{proposition}
The first non zero eigenvalue $\lambda_1$ of $-L$ satisfies the estimate
\[
\lambda_1 \ge \frac{\rho_1 \rho_2}{\frac{d-1}{d} \rho_2 +\kappa}.
\]
\end{proposition}

\begin{proof}
Let $f:\mathbb{M} \rightarrow \mathbb{R}$ be an eigenfunction corresponding to the eigenvalue $-\lambda_1$.
From the generalized curvature dimension inequality we know that for every $\nu >0$,
\begin{equation*}
\Gamma_{2}(f,f)+\nu \Gamma^Z_{2}(f,f) \ge \frac{1}{d} (Lf)^2 + \left( \rho_1 -\frac{\kappa}{\nu}\right)  \Gamma (f,f) + \rho_2 \Gamma^Z (f,f).
\end{equation*}
By integrating this inequality on the manifold $\mathbb{M}$, we obtain
\begin{equation*}
\int_{\mathbb{M}} \Gamma_{2}(f,f) d\mu+\nu\int_{\mathbb{M}} \Gamma^Z_{2}(f,f)d\mu \ge \frac{1}{d} \int_{\mathbb{M}}(Lf)^2 d\mu+ \left( \rho_1 -\frac{\kappa}{\nu}\right) \int_{\mathbb{M}} \Gamma (f,f) d\mu+ \rho_2\int_{\mathbb{M}} \Gamma^Z (f,f)d\mu.
\end{equation*}
Let us now recall that
\begin{equation*}
\Gamma_{2}(f,f) = \frac{1}{2}\big[L\Gamma(f,f) -2 \Gamma(f,
Lf)\big],
\end{equation*}
and
\begin{equation*}
\Gamma^Z_{2}(f,f) = \frac{1}{2}\big[L\Gamma^Z (f,f) - 2\Gamma^Z(f,
Lf)\big].
\end{equation*}
Therefore, by using $Lf=-\lambda_1f$ and integrating by parts in the
above inequality, we find
\begin{equation*}
\left( \lambda_1^2 -\frac{\lambda^2_1}{d}+\frac{\kappa \lambda_1}{\nu}  -\rho_1 \lambda_1 \right)
\int_{\mathbb{M}} f^2 d\mu \ge (\rho_2 -\nu \lambda_1) \int_{\mathbb{M}} \Gamma^Z (f,f)d\mu.
\end{equation*}
By choosing $\nu=\frac{\rho_2}{\lambda_1}$, we obtain the
inequality
\[
\lambda_1 \ge \frac{\rho_1 \rho_2}{\frac{d-1}{d} \rho_2 +\kappa}.
\]
\end{proof}

\begin{remark}
We note that when $\kappa = 0$, which corresponds to the Riemannian case,  we recover the classical theorem of
Lichnerowicz.
\end{remark}

\end{document}